\documentclass[10pt]{article}
\usepackage[english]{babel}										
\usepackage[utf8]{inputenc}										
\usepackage[T1]{fontenc}										
\usepackage{authblk}
\usepackage{amsmath,amsfonts,amssymb,amsthm,cancel,siunitx,
calculator,calc,mathtools,empheq,latexsym,mathrsfs}

\usepackage{subfig,epsfig,tikz,float}		            

\usepackage{booktabs,multicol,multirow,tabularx,array}          
\usepackage{hyperref}
\usepackage[skip=10pt plus1pt, indent=40pt]{parskip}
\def\opn#1#2{\def#1{\operatorname{#2}}} 
\opn\Ass{Ass}
\opn\G{G}
\opn\initial{in}
\title{\normalsize\bf%
\uppercase{Symbolic Powers and Symbolic Rees Algebras of Binomial Edge Ideals of Some Classes of Block Graphs}
}
\author{%
I. Jahani \ , \ Sh. Bayati\footnote{Corresponding Author} \ and \ F. Rahmati}

\newcommand{\Addresses}{{
  \bigskip
  \footnotesize

  \textsc{Department of Mathematics and Computer Science, Amirkabir University of Technology (Tehran Polytechnic), Iran}\par\nopagebreak
  \medskip \noindent
  \textit{E-mail address}: \texttt{imanjahani@aut.ac.ir}\\
  \textit{E-mail address}: \texttt{shamilabayati@gmail.com}\\
  \textit{E-mail address}: \texttt{frahmati@aut.ac.ir}

}}

\begin{document}

\date{}

\maketitle

\vspace{-0.5cm}

\bigskip
\noindent
{\small{\bf ABSTRACT.}
In this paper, we investigate some properties of symbolic powers and symbolic Rees algebras of binomial edge ideals associated with some classes of block graphs. First, it is shown that symbolic powers of binomial edge ideals of pendant cliques graphs coincide with the ordinary powers. Furthermore, we see that binomial edge ideals of a generalization of these graphs are symbolic $F$-split. Consequently, net-free generalized caterpillar graphs are also a class of block graphs with symbolic $F$-split binomial edge ideals. Finally, it turns out that symbolic Rees algebras of binomial edge ideals associated with these two classes, namely pendant cliques graphs and net-free generalized caterpillar graphs, are strongly $F$-regular.



\baselineskip=\normalbaselineskip
\newtheorem{definition}{Definition}[section]
\newtheorem{theorem}[definition]{Theorem}
\newtheorem{proposition}[definition]{Proposition}
\newtheorem{corollary}[definition]{Corollary}
\newtheorem{remark}[definition]{Remark}
\newtheorem{lemma}[definition]{Lemma}
\newtheorem{example}[definition]{Example}
\newtheorem{setup}[definition]{Setup}

\section{Introduction}\label{sec:1}
Let $I$ be an ideal of a Noetherian ring with minimal prime ideals $\mathfrak{p}_1,\ldots,\mathfrak{p}_r$. The $m$-th symbolic power of $I$ is $I^{(m)}=\mathfrak{q}_1\cap\cdots\cap\mathfrak{q}_r$ where $\mathfrak{q}_i$ is the $\mathfrak{p}_i$-primary component of $I^m$. Having the same associated primes as $I$ makes symbolic powers more important than ordinary powers in the frame of geometry. In this paper, we study symbolic powers of binomial edge ideals. 

Lots of efforts have been made to find out when symbolic and ordinary powers of ideals coincide. In the combinatorial context, there is a rich literature regarding the coincidence  of these powers of monomial ideals; see \cite{Alilooee,bahiano,bayati2014squarefree,DupVil.Edge,FranHaMermin,herzog2008standard,SimisVasVil,trung2011equality} for some related results. However, less is known about the equality of these powers when we consider binomial ideals associated with combinatorial structures. Ene and Herzog show in \cite[Theorem 3.3]{ene2018symbolic} that if $J_G$ is the binomial edge ideal of a simple connected graph $G$ with a squarefree initial ideal with respect to some monomial order, then the coincidence of ordinary and symbolic powers of that initial ideal implies the same property for $J_G$. Some known families of binomial edge ideals with equal ordinary and symbolic powers consist of Cohen-Macaulay binomial edge ideals of net-free block graphs \cite[Theorem 4.1]{ene2021powers},  binomial edge ideal of net-free generalized caterpillar graphs \cite[Theorem 3.11]{jahani2023equality}, and binomial edge ideal of complete multipartite graphs \cite[Theorem 4.3]{ohtani2013binomial}. In \cite{Taghipour}, a list of forbidden subgraphs is presented for the graphs whose parity binomial edge ideals have equal symbolic and ordinary powers. In Theorem~\ref{pc}, we show that pendant cliques graphs are another family of graphs whose binomial edge ideals have the same ordinary and symbolic powers.

Related to the problem of comparison of symbolic and ordinary powers, we next study the symbolic $F$-splitness property of binomial edge ideals. Symbolic $F$-split ideals are defined in \cite{msn} where it is shown there is a finite test for $F$-split ideals  which verifies whether their symbolic and ordinary powers coincide; see \cite[Theorem 4.7]{msn}.  So it is of interest to find out families of $F$-split ideals.
Moreover, symbolic Rees algebra and symbolic associated graded algebra associated with a symbolic $F$-split ideal are $F$-split \cite[Theorem 4.7]{msn}.   It is shown in \cite{ramirez2024symbolic} that the binomial edge ideals of the following graphs are symbolic $F$-split: complete multipartite graphs, caterpillar graphs, traceable graphs with unmixed binomial edge ideals, and those graphs whose binomial edge ideal has at most two associated primes. In this paper, we add one more family of graphs with symbolic $F$-split binomial edge ideals, namely generalized pendant cliques graphs; see Theorem~\ref{gpcfs}. As a result,  pendant cliques graphs  and  net-free generalized caterpillar graphs  are also $F$-split; see Corollary~\ref{cor.pcfs} and Corollary~\ref{cor.gcfs}. Our result is also a generalization of \cite[Theorem 4.6]{ramirez2024symbolic} about the $F$-splitness of caterpillar graphs.

Let $R$ denote a reduced Noetherian domain of positive prime
characteristic p. Consider  $R^{1/q}$ obtained by adjoining all $q$th roots of elements of R, regarded as an $R$-module in the
natural way. Suppose that   $R^{1/q}$                   is a finitely generated $R$-module. Let  $R^0$ denote the
complement of the union of the minimal primes of R . Strongly $F$-regularity is first defined by Hochster and Huneke in \cite{hochster1989tight} as follows: $R$  is strongly $F$-regular if for every   element $c\in R^0$,  there
exists $e\in \mathbb{N}$ such that the $R$-module homomorphism $R \to  R^{1/{p^e}}$ which sends $1$ to $c^{1/{p^e}}$ splits as a homomorphism of $R$-modules.  Some nice properties of strong $F$-regularity can be seen in   \cite[Theorem 3.1]{hochster1989tight}.  We refer to \cite{hochster} for a more detailed study regarding the theory of strongly $F$-regular rings. While strongly $F$-regularity has been widely studied, regarding the structures at the intersection between commutative algebra and combinatorics less is known about it. Strongly $F$-regularity of the Lov\'{a}sz-Saks-Schrijver ring associated with a  graph is studied in \cite{Tolo}.  It is shown in \cite{ramirez2024symbolic} that  symbolic Rees algebras associated with binomial edge ideals of the following families of graphs are strongly $F$-regular: complete multipartite graphs, closed graphs with unmixed binomial edge ideal, and graphs whose binomial edge ideal has at most two associated primes. We show the same result for pendant cliques graphs and net-free generalized caterpillar graphs; see Proposition~\ref{pendant.f-regular}  and  Proposition~\ref{gc.f-regular}.

\section{Preliminaries}
\label{1st}
In this section, we summarize some basic facts about  binomial edge ideals and symbolic $F$-splitness and fix some notations.

Throughout the paper, $G$ is a simple graph on the vertex set $V(G)$ with the edge set $E(G)$. Suppose that $v\in V(G)$. Then $G-v$ is the induced subgraph on $V(G)\setminus v$. If $v,w\in V(G)$, the distance between $v$ and $w$, denoted by $d(v,w) $, is the number of edges in a shortest path connecting them. 
By a \textbf{clique join on $G$ via an edge} $e=\{v_1,v_2\}\in E(G)$ we mean the graph obtained by attaching a complete graph $K_t$ to $G$ on $e$ for some $t\geq 3$, that is, a graph whose set of vertices is obtained by adding $t-2$  new vertices $u_1,\ldots,u_{t-2}$ to $V$, and $f$ is an edge  if $f$ is an edge of $G$ or $f=\{u_i,v_j\}$ for $i=1,\ldots ,{t-2}$ and $j=1,2$, or $f=\{u_i,u_j\}$ for $1\leq i<j\leq t-2$. Similarly, by a \textbf{clique join on $G$ via a vertex} $v\in V(G)$ we mean the graph obtained by attaching a complete graph $K_t$ to $G$ on $v$ for some $t\geq 2$, that is,  $t-1$ new vertices $u_1,\ldots,u_{t-1}$ are added to vertices and  edges $\{u_i,v\}$  for each $i=1,\ldots ,{t-1}$ and  $\{u_i,u_j\}$ for each $1\leq i <j \leq t-1$ are added to the set of edges. We denote these clique joins on $G$, respectively, by $G\sqcup_{e} K_t$ and $G\sqcup_{v} K_t$. A clique join of $K_2$ on $G$ via a vertex is called adding a \textbf{whisker} to $G$, and the newly added edge to $G$ is called a whisker.

A \textbf{generalized caterpillar graph} is a graph obtained by clique join of some complete graphs $K_{t_1},\ldots,K_{t_m}$ in succession on a path $P$ via pairwise distinct edges $e_1,\ldots,e_m$ of $P$ and finally, possibly adding a finite number of whiskers  to the vertices. A \textbf{pendant cliques graph} is a graph obtained by clique join of some complete graphs $K_{t_1},\ldots,K_{t_m}$ in succession on a path $P$ via pairwise distinct vertices  $v_1,\ldots,v_m$ of $P$. In this definition, if we allow an arbitrary finite number of clique joins via each vertex of $P$ and at most one clique join via each edge of $P$, then we call the obtained graph a \textbf{generalized pendant cliques graph}.  Figure \ref{fig:gencat},  Figure \ref{fig:pendantc}, and Figure~\ref{fig:genpen} are respectively examples of a generalized caterpillar graph, a pendant cliques graph, and a generalized pendant cliques graph. In these three definitions,  there might be more than one path which the conditions of the definition hold true for them. We call $P$  a \textbf{central path}  if among different  possible choices, the path $P$ is a longest one. To give an example, while the generalized pendant cliques graph in Figure~\ref{fig:genpen} can be obtained by clique joins via vertices and edges of the path $P:c,d,e,f,g$, it is not a central path. In this graph, for example, $P':b,c,d,e,f,g$ and $P'':a,c,d,e,f,g$ are central paths.

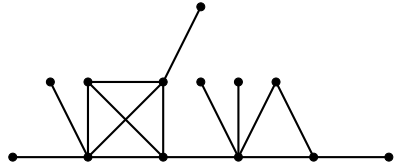
\begin{figure}[!h]
\centering
\begin{tikzpicture}
\draw[fill=black] (0,0) circle (1.5pt);
\draw[fill=black] (1,0) circle (1.5pt);
\draw[fill=black] (2,0) circle (1.5pt);
\draw[fill=black] (3,0) circle (1.5pt);
\draw[fill=black] (4,0) circle (1.5pt);
\draw[fill=black] (5,0) circle (1.5pt);
\draw[fill=black] (0.5,1) circle (1.5pt);
\draw[fill=black] (1,1) circle (1.5pt);
\draw[fill=black] (2,1) circle (1.5pt);
\draw[fill=black] (3,1) circle (1.5pt);
\draw[fill=black] (2.5,1) circle (1.5pt);
\draw[fill=black] (3.5,1) circle (1.5pt);
\draw[fill=black] (2.5,2) circle (1.5pt);
\draw[thick] (0,0) -- (1,0) -- (2,0) -- (3,0) -- (4,0) -- (5,0);
\draw[thick] (1,0) -- (0.5,1);
\draw[thick] (1,0) -- (1,1) -- (2,1) -- (2,0);
\draw[thick] (2,1) -- (2.5,2);
\draw[thick] (2,1) -- (1,0);
\draw[thick] (1,1) -- (2,0);
\draw[thick] (3,0) -- (2.5,1);
\draw[thick] (3,0) -- (3,1);
\draw[thick] (3,0) -- (3.5,1) -- (4,0);
\end{tikzpicture}
\caption{A generalized caterpillar graph}
\label{fig:gencat}
\end{figure}

\begin{figure}[!h]
\centering
\begin{tikzpicture}
\draw[fill=black] (0,0) circle (1.5pt);
\draw[fill=black] (1,0) circle (1.5pt);
\draw[fill=black] (2,0) circle (1.5pt);
\draw[fill=black] (3,0) circle (1.5pt);
\draw[fill=black] (4,0) circle (1.5pt);
\draw[fill=black] (5,0) circle (1.5pt);
\draw[fill=black] (-0.5,1) circle (1.5pt);
\draw[fill=black] (0.5,1) circle (1.5pt);
\draw[fill=black] (2,1) circle (1.5pt);
\draw[fill=black] (2.5,1) circle (1.5pt);
\draw[fill=black] (3.5,1) circle (1.5pt);
\draw[fill=black] (3,2) circle (1.5pt);
\draw[thick] (0,0) -- (1,0) -- (2,0) -- (3,0) -- (4,0) -- (5,0);
\draw[thick] (0,0) -- (-0.5,1);
\draw[thick] (0,0) -- (0.5,1);
\draw[thick] (-0.5,1) -- (0.5,1);
\draw[thick] (2,0) -- (2,1);
\draw[thick] (3,0) -- (2.5,1);
\draw[thick] (3,0) -- (3.5,1);
\draw[thick] (3,0) -- (3,2);
\draw[thick] (2.5,1) -- (3.5,1);
\draw[thick] (2.5,1) -- (3,2);
\draw[thick] (3.5,1) -- (3,2);
\end{tikzpicture}
\caption{A pendant cliques graph}
\label{fig:pendantc}
\end{figure}
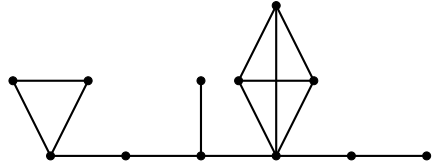

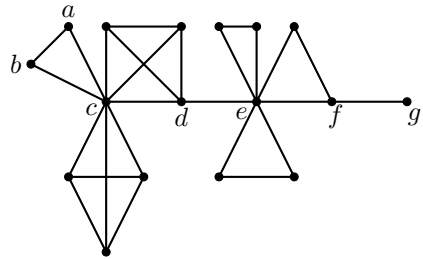
\begin{figure}[!h]
\centering
\begin{tikzpicture}
\draw[fill=black] (0,0.5) circle (1.5pt);
\draw[fill=black] (1,0) circle (1.5pt);
\draw[fill=black] (2,0) circle (1.5pt);
\draw[fill=black] (3,0) circle (1.5pt);
\draw[fill=black] (4,0) circle (1.5pt);
\draw[fill=black] (5,0) circle (1.5pt);
\draw[fill=black] (0.5,1) circle (1.5pt);
\draw[fill=black] (1,1) circle (1.5pt);
\draw[fill=black] (2,1) circle (1.5pt);
\draw[fill=black] (3,1) circle (1.5pt);
\draw[fill=black] (2.5,1) circle (1.5pt);
\draw[fill=black] (3.5,1) circle (1.5pt);

\draw[fill=black] (0.5,-1) circle (1.5pt);
\draw[fill=black] (1.5,-1) circle (1.5pt);
\draw[fill=black] (1,-2) circle (1.5pt);

\draw[fill=black] (2.5,-1) circle (1.5pt);
\draw[fill=black] (3.5,-1) circle (1.5pt);

\node at (0.5,1.2) {$a$};
\node at (-0.2,0.5) {$b$};
\node at (0.8,-0.1) {$c$};
\node at (2,-0.2) {$d$};
\node at (2.8,-0.15) {$e$};
\node at (4.05,-0.2) {$f$};
\node at (5.1,-0.2) {$g$};

\draw[thick] (0,0.5) -- (1,0) -- (2,0) -- (3,0) -- (4,0) -- (5,0);
\draw[thick] (1,0) -- (0.5,1);
\draw[thick] (0,0.5) -- (0.5,1);
\draw[thick] (1,0) -- (1,1) -- (2,1) -- (2,0);
\draw[thick] (2,1) -- (1,0);
\draw[thick] (1,1) -- (2,0);
\draw[thick] (3,0) -- (2.5,1) -- (3,1) -- (3,0);
\draw[thick] (3,0) -- (3.5,1) -- (4,0);

\draw[thick] (3,0) -- (2.5,-1) -- (3.5,-1) -- (3,0);
\draw[thick] (1,0) -- (0.5,-1) -- (1,-2) -- (1.5,-1) -- (1,0);
\draw[thick] (0.5,-1) -- (1.5,-1);
\draw[thick] (1,0) -- (1,-2);

\end{tikzpicture}
\caption{A generalized pendant cliques graph}
\label{fig:genpen}
\end{figure}

A graph $G$ is called a \textbf{net} if it is of the form of Figure \ref{fig:net}. We call a graph net-free if it does not have a net as an induced subgraph.
\begin{figure}[!h]
\centering
\begin{tikzpicture}
\draw[fill=black] (0.3,0.3) circle (1.5pt);
\draw[fill=black] (2.7,0.3) circle (1.5pt);
\draw[fill=black] (1,1) circle (1.5pt);
\draw[fill=black] (2,1) circle (1.5pt);
\draw[fill=black] (1.5,1.9) circle (1.5pt);
\draw[fill=black] (1.5,2.8) circle (1.5pt);
\draw[thick] (0.3,0.3) -- (1,1) -- (2,1) -- (2.7,0.3);
\draw[thick] (1,1) -- (1.5,1.9) -- (1.5,2.8);
\draw[thick] (1.5,1.9) -- (2,1);
\end{tikzpicture}
\caption{Net graph}
\label{fig:net}
\end{figure}
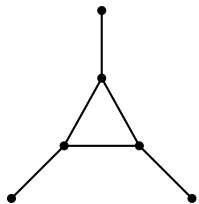

Consider a labeling $\sigma: V(G)\to [n]$ on the vertices of a graph $G$. The graph $G$ is called \textbf{closed} with respect to the  given labeling  if the following condition is satisfied: for all $\{i, j\}, \{k, l\} \in E(G)$ with $i < j$ and $k < l$, one has $\{j, l\} \in E(G)$ if $i = k$, and $\{i, k\} \in E(G)$ if $j = l$.
A simple graph $G$ is closed if there exists a labeling such that it is closed with respect to it.

\medskip

Let $\Delta$ be a simplicial complex on the vertex set [n]. A \textbf{cycle}  of length $s$ of $\Delta$ is an alternating sequence of distinct vertices and facets
\[
v_{1}, F_{1} , \ldots, v_{s} , F_{s} , v_{s+1} = v_{1}
\]
such that $v_{i} , v_{i+1} \in F_{i}$ for $i = 1, \ldots, s$. The cycle  is called \textbf{odd} (\textbf{even}) if $s$ is an odd (even) number. A cycle is \textbf{special} if it has no facet containing more than two vertices of the cycle.

Throughout the rest of this paper, suppose that $S = k[x_1,\ldots,x_n,y_1,\ldots,y_n]$ is the polynomial ring over a field $k$ with $2n$ variables. Associated with a monomial ideal $I \subseteq S$ with the minimal set of generators $\G(I)$, a simplicial complex $\Delta(I)$ is defined  whose facets are the sets
\begin{center}
$\{x_{i_1},\ldots ,x_{i_k}, y_{j_1},\ldots ,y_{j_l}\}$
\end{center}
with $x_{i_1}\ldots x_{i_k} y_{j_1}\ldots y_{i_l} \in \G(I)$.\\

Let $G$ be a graph on the vertex set $[n]$. The \textbf{binomial edge ideal} $J_{G}$ of $G$ is the ideal of $S$ generated by binomials $f_{i,j}=x_{i}y_{j} - y_{i}x_{j}$  which $\{i,j\}\in E(G)$. 
\begin{remark}\label{prime.description}
Let $G$ be a graph on the vertex set $[n]$.  For each subset $U\subseteq [n]$, a prime ideal $\mathfrak{p}_{U}(G)$ is defined in \cite{herzog2010binomial} as follows:
let $T=[n]\backslash U$, and  $G_{1},\ldots,G_{c(U)}$ be the connected components of the induced subgraph of $G$ on $T$. If for each $i$, $\tilde{G}_{i}$ is the complete graph on the vertex set $V(G_{i})$, then $\mathfrak{p}_U(G)$ is defined to be
\[
\mathfrak{p}_{U}(G) = (\{ x_{i},y_{i}\}_{i\in U} , J_{\tilde{G}_{1}} , \ldots , J_{\tilde{G}_{c(U)}}).
\]
\end{remark}
 By the following result, $\mathfrak{p}_{U}(G)$'s are exactly associated prime ideals of $J_G$.
\begin{theorem}{\em \cite[Theorem 3.2 and Corollary 3.9]{herzog2010binomial}}\label{binomial.primes}
Let $G$ be a graph on the vertex set $[n]$. Then 
\[
J_{G} = \underset{U\subset [n]}{\bigcap}\mathfrak{p}_{U}(G).
\]
Suppose, in addition, $G$ is connected. Then $\mathfrak{p}_{U}(G)$ is a minimal prime ideal of $J_G$ if and only if $U=\emptyset$, or $U\neq\emptyset$ and   $c(U \setminus \{i\}) < c(U)$ for each
$i\in U$.
\end{theorem}

 \begin{theorem}{\em \cite[Lemma 3.1]{herzog2010binomial}}\label{height}
Let $G$ be a graph on the vertex set $[n]$. Then 
 for each $U\subseteq [n]$
\[
\text{ht}(\mathfrak{p}_U) = |U| + n - c(U).
\]
\end{theorem}

Let $G$ be a graph on the vertex set $[n]$, and suppose that $i,j\in [n]$  and $i < j$. Then a path $\pi : i = i_{0} , i_{1} , \ldots , i_{r} = j$ in $G$  is called \textbf{admissible} if the following conditions hold:
\begin{enumerate}
\item either $i_k < i$ or $i_k > j$ for every $k = 1, \ldots , r - 1$;
\item for each proper subset $\{ j_{1} , \ldots , j_{s} \}$ of $\{i_{1} , \ldots , i_{r-1} \}$, the sequence $i, j_{1} , \ldots , j_{s} , j$ is not a path.
\end{enumerate}
In particular, all the edges $\{i, j\}$ of $G$, with $i < j$, are admissible paths from $i$ to $j$.
Now, let $\pi : i = i_{0} , i_{1} , \ldots , i_{r} = j$ be an admissible path in G. Associated with $\pi$ the following squarefree monomial is defined:
\begin{center}
$u_{\pi} := \underset{i_{k}>j}{\prod} x_{i_{k}} \underset{i_{l}<i}{\prod} y_{i_{l}}$.
\end{center}
The next theorem provides a reduced Gröbner basis for $J_{G}$  with respect to the  monomial order $<_{lex}$ which is the lexicographic term order induced by 
\[
x_1 >_{lex} \ldots >_{lex} x_{n} >_{lex} y_{1} >_{lex} \ldots >_{lex} y_{n}.
\]

\begin{theorem}{\em \cite[Theorem 2.1]{herzog2010binomial}}\label{thm5}
Let $G$ be a graph. Then the following set of binomials in $S$ is a reduced Gröbner basis of $J_{G}$ with respect to $<_{lex}$ as described above:
\begin{center}
$\mathcal{G} = \{ u_{\pi} f_{i j} : \pi \text{ is an admissible path from i to j} \}$.
\end{center}
\end{theorem}

The following result by Ene and Herzog provides a sufficient condition for a binomial edge ideal to have equal symbolic and ordinary powers and reduces the problem to squarefree monomial ideals to have the support of rich literature on squarefree monomial ideals with coincident symbolic and ordinary powers.

\begin{theorem}{\em \cite[Theorem 3.3]{ene2018symbolic}}\label{binomial.minimalprime}
Let $G$ be a connected graph on the vertex set $[n]$. If $\initial_{<_{lex}}(J_{G})$ is a normally torsion-free ideal, then $J_{G}^{(k)} = J_{G}^{k}$ for $k \geq 1$.
\end{theorem}

\medskip

Let $R$ be a reduced Noetherian ring of prime characteristic $p$, and consider $R^{1/p} $ obtained by adjoining all $p$th roots of elements of R.
Suppose that $\{I_n\}_{n \geq 0}$ is a sequence of ideals in $R$. The sequence $\{I_n\}_{n \geq 0}$ is an \textbf{$F$-split filtration} if the following holds:
\begin{enumerate}
\item $I_0 = R$.
\item $I_{n+1} \subseteq I_n$, for every $n \geq 0$.
\item $I_n I_m \subseteq I_{n+m}$, for every $n, m \geq 0$.
\item There is a splitting $\phi : R^{1/p} \to R$, of $R \subseteq R^{1/p}$, such that $\phi(((I_{np+1})^{1/p})) \subseteq I_{n+1}$ for every $n \geq 0$.
\end{enumerate}

An ideal $I\subseteq R$ is called \textbf{symbolic $F$-split} if $\{I^{(n)}\}_{n \geq 0}$ is an $F$-split filtration.

\medskip

The \textbf{symbolic Rees algebra} of an ideal $I\subseteq R$, denoted by $\mathscr{R}^s(I)$, is defined as follows:
\[
\mathscr{R}^s(I) = R \oplus I^{(1)}t \oplus I^{(2)}t^2 \oplus I^{(3)}t^3 \cdots \subseteq R[t].
\]   
 
Let $I$ be an ideal of a ring $R$ of prime characteristic $p$, and $q=p^e$ for some nonnegative integer $e$. The $q$th Frobenius power of $I$, denoted by $I^{[q]}$, defined to be 
\[
I^{[q]}=(f^q: f\in I).
\]

\section{Equality of symbolic and ordinary powers}\label{2nd} 

In this section, we study symbolic powers and symbolic Rees algebra of binomial edge ideals of some families of graphs.


\medskip

\begin{theorem}\label{pc}
If $G$ is a pendant cliques graph, then
\[
J_{G}^{t}=J_{G}^{(t)}
\]
for every $t \geq 1$.
\end{theorem}
\begin{proof}
First, we fix a central path $P:v_1,v_2,\ldots,v_r$ of $G$. 
We consider a labeling $\sigma : V(G) \rightarrow \{1,\ldots , n\}$ which is a one-to-one corresponding with the following properties:
\begin{enumerate}
\item For  distinct vertices $v_{i}$ and $v_{j}$, of $P$, if $i<j$, then
\begin{center}
$\sigma(v_i) < \sigma(v_j)$.
\end{center}
\item Let $K_{t_i}$ be the clique joined via $v_i\in V(P)$, and  $v$  be a vertex of $K_{t_i}$. Then $\sigma(v) > \sigma(v_i)$ if $i \neq 1$, and $\sigma(v) < \sigma(v_i)$ if $i = 1$.
\item Let  $v$ be a vertex of $K_{t_i}$, the clique joined via $v_i$, and $w$ be a vertex of $K_{t_j}$, the clique joined via $v_j$. If $i<j$, then
\begin{center}
$\sigma(v) < \sigma(w)$.
\end{center}
 \end{enumerate} 
By \cite[Theorem 2.5, Corollary 1.5, and Corollary 1.6]{herzog2008standard}, the facet ideal of a simplicial complex with no special odd cycle has coincident ordinary and symbolic powers.  So we only need to show that $\Delta(\initial_{<_{lex}}(J_{G}))$ has no special odd cycle.  Regarding the labeling described above and by Theorem~\ref{thm5}, the facets of $\Delta(\initial_{<_{lex}}(J_{G}))$ are $\{x_i,y_j\}$'s where $\{i<j\}$ is an edge of $G$, and $\{x_i,y_\ell,y_j\}$ where $\ell$ and $j$ are two vertices of $P$ with $\ell<i<j$ and $i$ is a vertex of the clique joined via $\ell$. For each $j$ on the central path $P$, set
\[V_{<j} = \{x_{i},y_{i}| i< j\}\]
and
\[V_{>j} = \{x_{i},y_{i}| i> j\}.\]
Then every cycle in $G$ that intersects both $V_{<j}$ and $V_{>j}$ must include $y_{j}$ twice and hence it is not special. As a result, we can only have special cycles in induced subcomplex on either $V_{<j}$ or $V_{>j}$ for each vertex $j$ on the central path $P$. So in order to prove that $\Delta(\initial_{<_{lex}}(J_{G}))$ has no special odd cycle, we may assume that $G$ is a complete graph, say $K_t$, with at most one whisker.  By labeling described above, the whisker would be either $\{t,t+1\}$ attached to the vertex $t$ of $K_t$, or $\{1,t+1\}$ attached to the vertex $1$ of $K_t$.

First, suppose that $G$ is only a complete graph $K_t$ with no whisker. By Theorem \ref{thm5}, facets of  $\Delta(\initial_{<_{lex}}(J_{G}))$ are only $G_{i,j}=\{x_{i},y_{j}\}$'s where $\{i<j\}$ is an edge of $K_t$. In this case, a special  cycle of $\Delta(\initial_{<_{lex}}(J_{G}))$ is of the form
\[
x_{i_1}G_{i_1, j_1} y_{j_1} G_{i_2, j_1}x_{i_2}\ldots y_{j_t}G_{i_1,j_t}x_{i_1}.
\]
In such a cycle  $x_{i_\ell}$'s exactly appear every other step, as well as $y_{j_\ell}$'s. Thus, the length of the cycle $C$ could be even.

Next, suppose that $G$ is the complete graph with a whisker. Two cases occur:
\begin{enumerate}
\item
First suppose that $G$ is the complete graph $K_t$ on the vertex set $\{1,\ldots,t\}$ with the whisker $\{t,t+1\}$. In  this case, by Theorem \ref{thm5}, facets of $\Delta(\initial_{<_{lex}}(J_{G}))$ are  $G_{i,j}=\{x_{i},y_{j}\}$ for each $1\leq i<j\leq t$, and $i=t$ and $j=t+1$. Since $G$ is a closed graph, $\Delta(\initial_{<_{lex}}(J_{G}))$ is a bipartite graph, as clarified in \cite{ene2018symbolic}, and does not have an odd cycle.

\item
Now suppose that $G$ is the complete graph $K_t$ on the vertex set $\{1,\ldots,t\}$ with the whisker $\{1,t+1\}$. By Theorem \ref{thm5}, facets of $\Delta(\initial_{<_{lex}}(J_{G}))$ are as follows: $G_{i,j}=\{x_{i},y_{j}\}$ for $1 \leq i < j \leq t$ or $i = 1$ and $j = t+1$, and $F_{i}=\{ y_{1},x_{i},y_{t+1} \}$ for $2 \leq i \leq t$.  

We claim that $G$ has no special odd cycle including $y_1$. Assume that $y_1$ appears in  a special odd cycle $C$ of $\Delta(\initial_{<_{lex}}(J_{G}))$.
Recall that $y_1$ only belongs to facets of type $F_i$. Regarding the facets and vertices appearing after $y_1$ in  $C$,  we distinguish two
cases:
\begin{enumerate}
\item
First, let the subsequence $F_i,y_1,F_{j},x_j$ appears in $C$ for some $2 \leq i \leq t$. Since $C$ is a special odd cycle, the vertex before $F_i$ can only be $x_i$. On the other hand, the other facets appearing in $C$ can be of type $G_{k,\ell}$. Hence $x_{k}$'s exactly appear every other step, as well as $y_{\ell}$'s. Thus, the length of the cycle $C$ is even.
\item
Next, suppose that the subsequence $F_i,y_{1},F_{j},y_{t+1}$ appears in $C$.  Each vertex of $F_i$ comes before it, the cycle is not special because three vertices of $F_i$ appear in the sequence.
\end{enumerate}

Hence $y_1$ does not appear in  a special odd cycle  of $\Delta(\initial_{<_{lex}}(J_{G}))$. So in a special cycle $C$ of $\Delta(\initial_{<_{lex}}(J_{G}))$, we can have $x_i$ and $y_{t+1}$ of $F_i$'s and $x_i$ and $y_j$ of $G_{i,j}$'s. So $x_{i}$'s exactly appear every other step in $C$, as well as $y_{j}$'s. Thus, the length of the cycle $C$ could be even.
\end{enumerate}
\end{proof}


Throughout the rest of this section, we assume that rings are of  prime characteristic $p$. The following result gives a sufficient condition for an ideal to be symbolic \(F\)-split.
\begin{lemma}\label{fs}
{\em \cite[Lemma 3.9]{ramirez2024symbolic}}
Let \(J\) be a homogeneous radical ideal in a polynomial ring over a field \(k\). Let \(J = \mathfrak{q_1} \cap \ldots \cap \mathfrak{q_{\ell}}\) , with \( \mathfrak{q_{i}} \in \min J \) and \(h_i = ht(q_i)\). Let \(\mathfrak{m}\) be the irrelevant ideal. If there is an \(f\) such that \( f \in \bigcap_{i} q_{i}^{h_i} \) and \( f^{p-1} \notin \mathfrak{m}^{[p]} \), then \(J\) is symbolic \(F\)-split.
\end{lemma}

\begin{setup}\label{gpcki}
Let $G$ be a generalized pendant cliques  graph, and fix a central path $P:v_1,\ldots,v_r$. Suppose that $p(v_i)$ denotes the number of cliques joined to $P$ via $v_i$. If
\[
K_{v_i,1}, \ldots , K_{v_i,p(v_i)}
\]
are cliques of size greater than one joined to $P$ via a vertex $v_i\in V(P)$, we set 
\[
\gamma_{v_i}=\sum_{\beta=1}^{p(v_i)} (|K_{v_i,\beta}|-1),
\]
which is the number of all vertices attached to $v_i$ by clique joins on $P$ via $v_i$. When a subset $\{s_1,\ldots,s_m\}\subseteq V(P)$ is specified, we denote $\gamma_{s_i}$ by $\gamma_i$ for short.
\end{setup}

\begin{remark}\label{label.generalized}
Let $G$ be a generalized pendant graph. Fix a central path $P=v_1,\ldots,v_r$ of $G$.  We consider a labeling $\sigma : V(G) \rightarrow \{1,\ldots , n\}$ which is a one-to-one corresponding with the following properties:
\begin{enumerate}
\item Let $K_t$ be a clique joined via a vertex $v_i\in V(P)$, and $K_s$ be a clique joined via the edge $\{v_i,v_{i+1}\}\in E(G)$. Suppose that $v\in V(K_t)\setminus \{v_i\}$ and $w\in V(K_s)\setminus \{v_i,v_{i+1}\}$. Then
\[
\sigma(v_i)<\sigma(v)<\sigma(w)<\sigma(v_{i+1}).
\]
\item Let $K_{t}$ and $K_{s}$ be two cliques joined via a vertex $v_i$ of  $P$, and suppose that $w_1,w_2\in V(K_{t})$ and $z\in V(K_{s})$. Then $\sigma(z)$ does not lie between $\sigma(w_1)$ and $\sigma(w_2)$. In other words, we label the vertices of cliques joined via a vertex one by one.
 \end{enumerate} 
\end{remark}

\begin{remark}\label{connected.components}
Let \( G \) be a generalized pendant cliques graph on $[n]$ labeled as described in Remark~\ref{label.generalized}, and suppose that $P:v_1,\ldots,v_r$ is  a central path of $G$. Assume Setup~\ref{gpcki}. By Theorem~\ref{binomial.primes},  if $\mathfrak{p}_{U}(G)$ is a minimal prime ideal of $J_G$ then $U=\emptyset$ or $U=\{1<s_1<\cdots<s_m<n\}\subseteq V(P)$. Using the inductive argument, we can find the connected components of the induced subgraph of $G$ on  $[n]\setminus U$ when $U\neq \emptyset$. Assume that $s_1,\ldots, s_{i}$ has been removed from $G$ for some $i\geq 1$ before, and now remove $s_{i+1}$. Then the following connected components would appear:
\begin{itemize}
\item The induced subgraph on the interval $[s_i+\gamma_i+1,s_{i+1}-1]$ of vertices if the interval is non-empty.
\item The induced subgraph  $K_t-s_{i+1}$ for each clique $K_t$ of size $t>1$ joined via $s_{i+1}$.
\item The induced subgraph on the interval $[s_{i+1}+\gamma_{i+1}+1,n]$ of vertices. 
\end{itemize}
Notice that the interval in the last case is always non-empty. In fact, if $s_m+\gamma_m=n$, then $s_m$ is the last vertex of $P$ and $n$ is a vertex of some clique $H=K_t$ joined via $s_m$. So one can find a path $\tilde{P}$ which is $P$  continued by $n$, that is, 
\[
\tilde{P}:v_1,\ldots,v_r,n.
\]
To obtain $G$ by clique joins, all clique joins via vertices and edges in $\tilde{P}$ are the same as $P$ except  $H=K_t$. In $\tilde{P}$, $H-n=K_t-n$ is joined via the edge $\{s_m,n\}\in E(\tilde{P})$. Thus we find $G$ by clique joins on a path longer than $P$; a contradiction to the fact that the central path $P$ is a longest path among possible choices. 

\end{remark}

\medskip

By Remark~\ref{connected.components}, we have the following result. In this result, assume Setup~\ref{gpcki}. 
\begin{lemma}\label{number.cc}
Let \( G \) be a generalized pendant cliques graph on $[n]$ labeled as described in Remark~\ref{label.generalized}, and suppose that $P$ is  a central path of $G$. If $U=\{1<s_1<\cdots<s_m<n\}\subseteq V(P)$, then $c(U)$, the number of  connected components of the induced subgraph on  $[n]\setminus U$, is equal to
\[
c(S) =  \sum_{i=0}^{m} \alpha_i + \sum_{i=1}^{m} p(s_i),
\]
where for  $i=1,\ldots ,m-1$, we set $\alpha_i=1$ if the interval  $[s_i+\gamma_i+1,s_{i+1}-1]$ is non-empty, and $\alpha_i=0$ otherwise. We also set $\alpha_0=\alpha_m=1$.
\end{lemma}
\begin{proof}
We prove the lemma by induction on $m$. First, suppose that $m=1$. Then $\alpha_m=\alpha_1=1$ as set in the statement. Clearly, by removing $s_1$, we have a connected component on the interval $[1,s_1-1]$ of vertices, $p(s_1)$ components by cliques joined via $s_1$, and one component on $[s_1+\gamma_1+1,n]$. So in this case, one has
\[
c(S) =  2+  p(s_1)= \alpha_0+\alpha_1+p(s_1).
\] 

Assume that $m\geq 2$. By induction hypothesis, removing $U=\{1<s_1<\cdots<s_{m-1}\}\subseteq V(P)$ gives 
$ \alpha_0+\alpha_1+\ldots+\alpha_{m-2}+1 + \sum_{i=1}^{m-1} p(s_i)$
 components. Hence by Remark~\ref{number.cc}, after removing $s_m<n$, one of the previous components, that is, the component on $[s_{m-1}+\gamma_{m-1}+1, n]$ becomes $\alpha_{m-1}+p(s_m)+1$ new components. Hence $\alpha_{m-1}+p(s_m)$ is added to the number of components in the induction hypothesis.
\end{proof}
\medskip

\begin{theorem}\label{gpcfs}
Let \( G \) be a generalized pendant cliques graph. Then \( J_G \) is symbolic \( F \)-split.
\end{theorem}
\begin{proof}
Let $G$ be on the vertex set $[n]$ with labeling described in Remark~\ref{label.generalized}, and fix a central path $P$. We set
\[ 
f = y_1 f_{1,2} \ldots f_{n-1,n} x_n.
\]
We are going to show that $f$ is the desired element in Lemma~\ref{fs}. For this purpose, we first show that  \( f \in \mathfrak{p}_U^{h} \) for each minimal prime \( \mathfrak{p}_U \) of \( J_G \) where \( h = \text{ht}(\mathfrak{p}_U) \). 
By Remark~\ref{connected.components}, we distinguish two cases \( U = \emptyset \) and \( U \neq \emptyset \).

First, suppose that \( U = \emptyset \). Then by Remark~\ref{prime.description},  \( \mathfrak{p}_U \) is the binomial edge ideal of the complete graph $K_n$. In particular. \( f_{i, i+1} \in \mathfrak{p}_U \) for every \( i \in [n-1] \). Hence, \( f \in \mathfrak{p}_U^{n-1} \). On the other hand, $n-1$ is  $\text{ht}(\mathfrak{p}_U)$ by Theorem~\ref{height}. Hence \( f \in \mathfrak{p}_U^h \).

Next, suppose that \( U \neq \emptyset \). Then by Remark~\ref{connected.components}  $U$ is of the form $ \{1<s_1< s_2< \cdots< s_m<n\}\subseteq V(P)$.  Consider $p(s_i)$ and $\gamma_i=\gamma_{s_i}$ as Setup~\ref{gpcki}.
By inductive argument in Remark~\ref{connected.components}, finally, one has the following connected components by  removing $U=\{1<s_1<\cdots<s_m<n\}\subseteq V(P)$:
\begin{enumerate}
\item[(C1)] The induced subgraph on the interval $T=[1,s_1-1]$.
\item[(C2)] The induced subgraph on the interval $T=[s_i+\gamma_i+1,s_{i+1}-1]$  for each $i=1,\ldots,m-1$ if the interval is non-empty.
\item[(C3)] The induced subgraph  $K_{t}-s_{i}$ for each clique $K_{t}$ with $t\geq 2$ joined via $s_{i}$ and each $i=1,\ldots,m$. In this case, set $T=V(K_{t}-s_{i})$.
\item[(C4)] The induced subgraph on the interval $T=[s_{m}+\gamma_{m}+1,n]$.
\end{enumerate}
By Remark~\ref{prime.description}, each of these connected components gives a summand $J(K_T)$ of $\mathfrak{p}_U$ where $J(K_T)$ is the binomial edge ideal of the complete graph $K_T$ on the vertex set $T$. More precisely,  If $\mathfrak{T}$ is the set of all $T$'s described in (C1-C4), then
\[
\mathfrak{p}_U =  (x_s, y_s \mid s \in U)+\sum_{T\in \mathfrak{T}} J(K_T).
\]
Associated with each set $T\in \mathfrak{T}$, we define a polynomial $g_T$ as follows: $g_T$ is $1$ when $|T|\leq 1$, and otherwise 
\[
g_T=\prod_{\{j,j+1\}\subseteq T} f_{j,j+1}.
\] 
Now, set the polynomial $g$ to be
\[
g=( \prod_{i=1}^{m} f_{s_{i}-1,s_{i}}f_{s_{i},s_{i}+1})(\prod_{T\in\mathfrak{T}}g_T).
\]
The polynomial $g$ divides $f$. So in order to show that \( f \in \mathfrak{p}_U^h \), we are going to show that \( g \in \mathfrak{p}_U^h \). For this purpose, first notice that
\[
 f_{s_{i}-1,s_{i}}f_{s_{i},s_{i}+1}\in (x_s, y_s \mid s \in U)^2.
 \]
 Hence
 \[
 \prod_{i=1}^{m} f_{s_{i}-1,s_{i}}f_{s_{i},s_{i}+1}\in (x_s, y_s \mid s \in U)^{2m}\subseteq \mathfrak{p}_U^{2m}.
 \]
 Next, we consider the factors $g_T$'s. If $T\neq \emptyset$, then
 \[
 g_T\in J(K_T)^{\left| T \right| -1}\subseteq \mathfrak{p}_U^{\left| T \right| -1},
 \]
and if $T= \emptyset$, then $g_T=1\in\mathfrak{p}_U^0 $. So we can conclude that 
\[
 g_T\in \mathfrak{p}_U^{\left| T \right| -\alpha_T}
\]
where $\alpha_T=1$ if $T\neq \emptyset$, and $\alpha_T=0$ otherwise. 
Hence, 
\[
g\in \mathfrak{p}_U^b
\]
where $b=2m+\sum_{T\in \mathfrak{T}} (\left| T \right|-\alpha_T)$.
Now, we compute $\sum (\left| T \right|-\alpha_T)$ over $T$'s described  in (C1-C4).
\begin{enumerate}
\item[(D1)] For $T$ described in (C1), $\alpha_T=1$ because $1<s_1$ and consequently, $T\neq \emptyset$. So $\left| T\right|-\alpha_T=s_1-1-\alpha_0$ where $\alpha_0$ is the number described in Lemma~\ref{number.cc}.
\item[(D2)] For $T$'s described in (C2), $\left| T\right|-\alpha_T=s_{i+1}-s_i-\gamma_i-1-\alpha_i$ where $i=1,\ldots,m-1$, and $\alpha_i$ is the number described in Lemma~\ref{number.cc}. 
\item[(D3)] For $T=V(K_t-s_{i})$'s described in (C3), one has $\left| T\right|-\alpha_T=t-1-1$ where $i=1,\ldots,m$ and $K_t$ is a clique joined via $s_i$. Notice that in this case we always have $\alpha_T=1$. We conclude that sum of $(\left| T\right|-\alpha_T)$'s over all $T$'s described in (C3) is equal to
\[
\sum_{i=1}^{m}\gamma_i-\sum_{i=1}^{m} p(s_i),
\]
where $\gamma_i$ and $p(s_i)$ are as Setup~\ref{gpcki}.
\item[(D4)] For $T$ described in (C4), $\alpha_T=1$ because $T\neq \emptyset$ as clarified in Remark~\ref{connected.components}. Thus $\left| T\right|-\alpha_T=n-s_{m}-\gamma_{m}-\alpha_m$ where $\alpha_m$ is the number described in Lemma~\ref{number.cc}.
\end{enumerate}
Finally, we have
\begin{equation*}
\begin{aligned}
b&=2m+\sum_{T\in \mathfrak{T}} (\left| T \right|-\alpha_T)\\
  &=2m+(s_1-1-\alpha_0)+\sum_{i=1}^{m-1}(s_{i+1}-s_i-\gamma_i-1-\alpha_i)+\\
  & \,\,\,\,\,\,\, \sum_{i=1}^{m}\gamma_i-\sum_{i=1}^{m} p(s_i)+(n-s_{m}-\gamma_{m}-\alpha_m)\\
  &=m+n-\sum_{i=0}^{m}\alpha_i-\sum_{i=1}^{m} p(s_i).
\end{aligned}
\end{equation*}
But this is exactly $h=\text{ht}(\mathfrak{p}_U)$ by Theorem~\ref{height} and Lemma~\ref{number.cc}. Hence $g$ and consequently $f$ belong to $\mathfrak{p}_U^h$, as desired.

On the other hand, one can see that \(f^{p-1} \notin \mathfrak{m}^{[p]}\) where \(\mathfrak{m}\) is the irrelevant ideal.  This is a consequence of the fact that $f$ has a squarefree term \(u=y_1x_1y_2x_2y_3 \ldots x_{n-1}y_nx_n\) and none of the elements of \(\mathfrak{m}^{[p]}\) divide \(u^{p-1}\). 

Thus Lemma \ref{fs} implies that \(J_G\) is symbolic \(F\)-split. 
\end{proof}

\medskip

\begin{corollary}\label{cor.pcfs}
Let \( G \) be a pendant cliques graph. Then \( J_G \) is symbolic \( F \)-split.
\end{corollary}
\begin{corollary}\label{cor.gcfs}
Let \( G \) be a net-free generalized caterpillar graph. Then \( J_G \) is symbolic \( F \)-split.
\end{corollary}

\medskip

By Theorem~\ref{gpcfs}, in particular, we have the following result:
\begin{corollary} {\em \cite[Theorem 4.6]{ramirez2024symbolic}}\label{cor.cat}
Let \( G \) be a  caterpillar graph. Then \( J_G \) is symbolic \( F \)-split.
\end{corollary}

\medskip

We now turn our attention to the problem  whether the symbolic Rees Algebra associated with a binomial edge ideal is strongly \( F \)-regular.  The following result by Ram{\'\i}rez-Moreno gives a sufficient condition for symbolic Rees Algebras to be strongly \( F \)-regular. Recall that $p$ is characteristic of ring $S$.
\begin{lemma}\label{sfrlemma}
{\em \cite[Lemma 5.1]{ramirez2024symbolic}}
Let \( J \) be a homogeneous radical ideal of a polynomial ring \( S \) over a field \( k \) with the irrelevant ideal \( \mathfrak{m} \). Suppose that \( \mathfrak{q}_1, \mathfrak{q}_2, \ldots, \mathfrak{q}_l \) are the minimal prime ideals of \( J \). For \( i \in [l] \), set \( h_i = \operatorname{ht}(\mathfrak{q}_i) \), and suppose that \( f \in S \) such that \( f = cq \) with \( c, q \in S \) and \( c \ne 0 \). Suppose the following holds:
\begin{enumerate}
\item \( q \in \bigcap_i \mathfrak{q}_i^{h_i - 1} \);
\item \( f^{p-1} \notin \mathfrak{m}^{[p]} \);
\item \( (\mathscr{R}^s(J))_c \) is strongly \( F \)-regular;
\item \( \mathscr{R}^s(J) \) is Noetherian.
\end{enumerate}
Then \( \mathscr{R}^s(J) \) is strongly \( F \)-regular.
\end{lemma}

\medskip

In \cite[Theorem 5.5]{ramirez2024symbolic}, it is shown that if $G$ is a closed graph such that $J_G$ is unmixed, then  by setting $f = y_1 f_{1,2} \cdots f_{n-1,n} x_n$ and $c=f_{1,2}$, in Lemma~\ref{sfrlemma}, one obtains that \( \mathscr{R}^{s}(J_G) \) is strongly $F$-regular. One can see that the same argument can be applied to  pendant cliques graphs and net-free generalized caterpillar graphs as follows:
\begin{proposition}\label{pendant.f-regular}
Let \( G \) be a  pendant cliques graph.  Then \( \mathscr{R}^{s}(J_G) \) is strongly \(F\)-regular.
\end{proposition}
\begin{proof}
Let $G$ be on the vertex set $[n]$. We consider the labeling described in the proof of Theorem~\ref{pc}. Set \( f = y_1 f_{1,2} \cdots f_{n-1,n} x_n \), \( q = f / f_{1,2} \), and $c=f_{1,2}$. We show that these elements satisfy the conditions of Lemma~\ref{sfrlemma}. 
\begin{enumerate}
\item Suppose that $\mathfrak{p}_U$ be a minimal prime of $J_G$ associated with $U=\{1<s_1<\cdots<s_{m}<n\}\subseteq V(P)$. Assume $T$, $g_T$, and $g$ are as determined in the proof of Theorem~\ref{gpcfs}. In particular,
\[
g=( \prod_{i=1}^{m} f_{s_{i}-1,s_{i}}f_{s_{i},s_{i}+1})(\prod_{T\in\mathfrak{T}}g_T).
\]
By the argument we have in that proof, if $s_1=2$, 
\[
( \prod_{i=2}^{m} f_{s_{i}-1,s_{i}}f_{s_{i},s_{i}+1})(\prod_{T\in\mathfrak{T}}g_T)  \in \mathfrak{p}_U^{h-2}.
\]
On the other hand,
\[
f_{2,3}\in (x_i,y_i| i\in U)\subseteq \mathfrak{p}_U;
\]
See Remark~\ref{prime.description}. Hence $g/f_{1,2}\in \mathfrak{p}_U^{h-1}$. Thus $f/f_{1,2}\in \mathfrak{p}_U^{h-1}$.

Next, suppose that $s_1>2$. Then $f_{1,2}$ divides $g_T$ for $T=[1,s_1-1]$, and $g/f_{1,2}\in \mathfrak{p}_U^{h-1}$. Hence $f/f_{1,2}\in \mathfrak{p}_U^{h-1}$.

\item \( f^{p-1} \notin \mathfrak{m}^{[p]}\) as we have seen in proof of Theorem~\ref{gpcfs}.

\item By labeling of $G$, $\{1,2\}\in E(G)$. So $f_{1,2}\in J_G$. This implies that ${J_{G}}_{f_{1,2}}=S_{f_{1,2}}$. Thus
\( {\mathscr{R}^{s}(J_G)}_{f_{1,2}}=S_{f_{1,2}}[t] \). The regular ring $S$ is strongly $F$-regular, and as a result, $S_{f_{1,2}}$ is  also strongly $F$-regular; see \cite[Theorem 3.1]{hochster1989tight}. We deduce that \( {\mathscr{R}^{s}(J_G)}_{f_{1,2}}=S_{f_{1,2}}[t] \) is strongly $F$-regular.

\item Since the symbolic and ordinary powers of $J_G$ coincide by Theorem~\ref{pc}, symbolic Rees algebra and Rees algebra of $J_G$ also coincide. Recall that Rees algebra $S[J_Gt]$ of $J_G$ is Noetherian.  

\medskip

Now by Lemma \ref{sfrlemma}, we conclude that  \( \mathscr{R}^{s}(J_G) \) is strongly \(F\)-regular.
\end{enumerate}

\end{proof}

\begin{proposition}\label{gc.f-regular}
Let \( G \) be a net-free generalized caterpillar graph. Then \( \mathscr{R}^{s}(J_G) \) is strongly \(F\)-regular.
\end{proposition}
\begin{proof}
Let $G$ be on the vertex set $[n]$. We consider the labeling described in Remark~\ref{label.generalized}. One can see that in the case of   net-free generalized caterpillar graphs, this labeling is exactly that one which is used in \cite[Lemma 3.10, Theorem 3.11]{jahani2023equality} to show that symboilc powers and ordinary powers of net-free generalized caterpillar graphs coincide. Hence symbolic Rees algebra and Rees algebra of $J_G$ also coincide and \( \mathscr{R}^{s}(J_G) \) is Noetherian.

Set \( f = y_1 f_{1,2} \cdots f_{n-1,n} x_n \), \( q = f / f_{1,2} \), and $c=f_{1,2}$. One can see that these elements satisfy  conditions 1-3 of Lemma~\ref{sfrlemma} by applying the same argument used in the proof of Proposition~\ref{pendant.f-regular}.
\end{proof}

\newpage
{}
\Addresses
\end{document}